\documentclass{amsart}
\usepackage{amsthm,xcolor, bm}
\usepackage{latexsym,amssymb,amsfonts,amsmath,graphicx, url,mathtools}
\usepackage[mathscr]{eucal}
\usepackage[vcentermath,enableskew]{youngtab}
\usepackage{color}
\usepackage[noadjust]{cite} % the spacing of package cite solves a problem we don't have and introduces a problem we do.
\usepackage{tikz,hyperref}
\usepackage[margin=1in]{geometry}
\usepackage{enumerate,textcomp, wasysym, enumitem}

\newtheorem{theorem}{Theorem}[section]
\newtheorem{proposition}[theorem]{Proposition}
\newtheorem{lemma}[theorem]{Lemma}
\newtheorem{conjecture}[theorem]{Conjecture}
\newtheorem{corollary}[theorem]{Corollary}

\theoremstyle{definition}
\newtheorem{definition}[theorem]{Definition}
\newtheorem{example}[theorem]{Example}

\newtheoremstyle{named}{}{}{\itshape}{}{\bfseries}{.}{.5em}{#1 \thmnote{#3}}
\theoremstyle{named}
\newtheorem*{namedtheorem}{Theorem}

\title[Principal specialization of dual characters of flagged Weyl~modules]{Principal specialization of dual characters \\ of flagged Weyl~modules}

\author{Karola M\'esz\'aros}
\address{Karola M\'esz\'aros, Department of Mathematics, Cornell University, Ithaca, NY 14853.  \newline\textup{karola@math.cornell.edu}
}

\author{Avery St.~Dizier}
\address{Avery St.~Dizier, Department of Mathematics, University of Illinois at Urbana-Champaign, Urbana, IL 61801.  \newline\textup{stdizie2@illinois.edu}}

\author{Arthur Tanjaya}
\address{Arthur Tanjaya, Department of Mathematics, Cornell University, Ithaca NY 14853.  \newline\textup{amt333@cornell.edu}
}

\thanks{Karola M\'esz\'aros is partially  supported by  CAREER NSF Grant DMS-1847284. Avery St.~Dizier received support from NSF Grant DMS-2002079.}

\begin{document}
	\maketitle
	\begin{abstract}
%		Stanley (2017) conjectured that the principal specialization of the Schubert polynomial of a permutation $w$ equals two if and only if $w$ exactly one 132-pattern.
%%		 $\S_w(1, \ldots, 1)$ 
%		Shortly thereafter, Weigandt proved Stanley's conjecture and generalized it, showing that the principal specialization of the Schubert polynomial of $w$ is bounded below by one plus the number of $132$-patterns in $w$.
%%		using the combinatorial pipe dream model for Schubert polynomials. 
%		We prove a generalization of Weigandt's result to dual characters of flagged Weyl modules, a family of polynomials containing all Schubert and key polynomials.
		Schur polynomials are special cases of Schubert polynomials, which in turn are special cases of dual characters of flagged Weyl modules. The principal specialization of Schur and Schubert polynomials has a long history, with Macdonald famously expressing the principal specialization of any Schubert polynomial in terms of reduced words. We study the principal specialization of dual characters of flagged Weyl~modules. Our result yields an alternative proof of a conjecture of Stanley about  the principal specialization of Schubert polynomials, originally proved by Weigandt.
%		the principal specialization of the Schubert polynomial of a permutation $w$
%		is bounded below by one plus the number of $132$-patterns in~$w$.  
	 \end{abstract}
		
	\section{Introduction}
	Schubert polynomials $\mathfrak{S}_w$ were introduced by Lascoux and Sch\"utzenberger in \cite{LS1} as distinguished polynomial representatives for the cohomology classes of Schubert cycles in the flag variety. Schubert polynomials generalize Schur polynomials, a classical basis of the ring of symmetric functions. 
	
	The principal specialization of Schur polynomials has a long history: $s_{\lambda}(1,\ldots,1)$ counts the number of semistandard Young tableaux of shape $\lambda$, a number famously enumerated by the hook-content formula, see for instance \cite{EC2}. Macdonald \cite[Eq. 6.11]{macdonald} famously expressed the principal specialization $\mathfrak{S}_{w}(1,\ldots,1)$ of the Schubert polynomial $\mathfrak{S}_{w}$  in terms of the reduced words of the permutation $w$. Fomin and Kirillov \cite{reduced} placed this expression in the context of plane partitions for dominant permutations, while after two decades Billey et al. \cite{bijmac} provided a combinatorial proof.  Principal specialization of Schubert polynomials has inspired a flurry of recent interest \cite{shenanigans, anna132,yibo, arthur, pak}.
	A major catalyst for the current line of study into $\mathfrak{S}_w(1,\ldots,1)$ is the following result of Weigandt, which generalizes a conjecture of Stanley (\cite[Conjecture 4.1]{shenanigans}).
	\begin{theorem}[{\cite[Theorem 1.1]{anna132}}]
		\label{thm:schubprinspec}
		For any permutation $w\in S_n$, 
		\[\mathfrak{S}_w(1,\ldots,1)\geq 1+p_{132}(w), \]
		where $p_{132}(w)$ is the number of 132-patterns in $w$.
	\end{theorem}
 
 	Weigandt's proof of Theorem \ref{thm:schubprinspec} works by exploiting the structure of pipe dreams, one of the earliest combinatorial models for Schubert polynomials \cite{laddermoves,FKschub}. We give an alternative proof of Theorem \ref{thm:schubprinspec} by generalizing its statement to the setting of dual characters of flagged Weyl modules of diagrams: 
 
 	\begin{theorem} 
 		\label{thm:main} For any diagram $D$, the dual character $\chi_D$  of the flagged Weyl module of $D$ satisfies 
 		\[\chi_D(1, \dots, 1) \geq \textup{rank}(D) + 1.\]
 	\end{theorem}
 
	We show in Corollary \ref{cor:annarelation} that Theorem \ref{thm:main} specializes to Theorem \ref{thm:schubprinspec}. Additionally, Theorem \ref{thm:main} implies an analogous result for key polynomials (Corollary \ref{cor:keys}).
%	Schubert polynomials are known to occur as dual characters $\chi_D$ whenever $D=D(w)$ is the Rothe diagram of a permutation $w$, and the number of $132$-patterns in $w$ equals the rank of $D(w)$.
	
	\subsection*{Outline of this paper}
	In Section \ref{sec:background} we  define dual characters of flagged Weyl modules of diagrams, and we provide necessary background. In Section \ref{sec:lower}, we define the rank of a diagram and prove Theorem \ref{thm:main}. We characterize the case of equality in Theorem \ref{thm:main} and connect to zero-one polynomials. We conclude in Section \ref{sec:upper} by describing a simple upper bound version of Theorem \ref{thm:main}, and conjecturing a characterization for when equality holds.
 
	\section{Background}
	\label{sec:background}
	We first define flagged Weyl modules and their dual characters. We then recall the definition of Schubert polynomials and the connection between Schubert polynomials and dual characters.  The exposition of this section follows that of \cite{zeroone}.
	 
	By a \emph{diagram}, we mean a sequence $D=(C_1,C_2,\ldots,C_n)$ of finite subsets of $[n]$, called the \emph{columns} of $D$. We interchangeably think of $D\subseteq [n]\times [n]$ as a collection of boxes $(i,j)$ in a grid, viewing an element $i\in C_j$ as a box in row $i$ and column $j$ of the grid. When we draw diagrams, we read the indices as in a matrix: $i$ increases top-to-bottom and $j$ increases left-to-right. 
	%Two diagrams $D$ and $D'$ are called \emph{column-equivalent} if one is obtained from the other by reordering nonempty columns and adding or removing any number of empty columns. For a column $C\subseteq [n]$, let the \emph{multiplicity} $\mathrm{mult}_D(C)$ be the number of columns of $D$ which are equal to $C$. The sum of diagrams, denoted $D\oplus D'$, is constructed by concatenating the lists of columns; graphically this means placing $D'$ to the right of $D$. 

	Let $G=\mathrm{GL}(n,\mathbb{C})$ be the group of $n\times n$ invertible matrices over $\mathbb{C}$ and $B$ be the subgroup of $G$ consisting of the $n\times n$ upper-triangular matrices. The flagged Weyl module is a representation of $B$ associated to a diagram $D$. The flagged Weyl module of $D$ will be denoted by $\mathcal{M}_D$. We will use the construction of $\mathcal{M}_D$ in terms of determinants given in \cite{magyar}.
	
	Denote by $Y$ the $n\times n$ matrix with indeterminates $y_{ij}$ in the upper-triangular positions and zeros elsewhere. Let $\mathbb{C}[Y]$ be the polynomial ring in the indeterminates $\{y_{ij}\}_{i\leq j}$. Note that $B$ acts on $\mathbb{C}[Y]$ on the right via left translation: if $f(Y)\in \mathbb{C}[Y]$, then a matrix $b\in B$ acts on $f$ by $f(Y)\cdot b=f(b^{-1}Y)$. For any $R,S\subseteq [n]$, let $Y_S^R$ be the submatrix of $Y$ obtained by restricting to rows $R$ and columns $S$.
	
	For $R,S\subseteq [n]$, we say $R\leq S$ if $\#R=\#S$ and the $k$\/th least element of $R$ does not exceed the $k$\/th least element of $S$ for each $k$. For any diagrams $C=(C_1,\ldots, C_n)$ and $D=(D_1,\ldots, D_n)$, we say $C\leq D$ if $C_j\leq D_j$ for all $j\in[n]$.

	\begin{definition}
		For a diagram $D=(D_1,\ldots, D_n)$, the \emph{flagged Weyl module} $\mathcal{M}_D$ is defined by
		\[\mathcal{M}_D=\mathrm{Span}_\mathbb{C}\left\{\prod_{j=1}^{n}\det\left(Y_{D_j}^{C_j}\right)\ \middle|\   C\leq D \right\}. \]
		$\mathcal{M}_D$ is a $B$-module with the action inherited from the action of $B$ on $\mathbb{C}[Y]$. 
	\end{definition}
	Note that since $Y$ is upper-triangular, the condition $C\leq D$ is technically unnecessary since $\det\left(Y_{D_j}^{C_j}\right)=0$ unless $C_j\leq D_j$. Conversely, if $C_j\leq D_j$, then $\det\left(Y_{D_j}^{C_j}\right)\neq 0$. 
	
	For any $B$-module $N$, the \emph{character} of $N$ is defined by $\mathrm{char}(N)(x_1,\ldots,x_n)=\mathrm{tr}\left(X:N\to N\right)$ where $X$ is the diagonal matrix with diagonal entries $x_1,\ldots,x_n$, and $X$ is viewed as a linear map from $N$ to $N$ via the $B$-action. Define the \emph{dual character} of $N$ to be the character of the dual module $N^*$:
	\begin{align*}
		\mathrm{char}^*(N)(x_1,\ldots,x_n)&=\mathrm{tr}\left(X:N^*\to N^*\right) \\
		&=\mathrm{char}(N)(x_1^{-1},\ldots,x_n^{-1}).
	\end{align*}
	\begin{definition}
		For a diagram $D\subseteq [n]\times [n]$, let $\chi_D=\chi_D(x_1,\ldots,x_n)$ be the dual character 
		\[\chi_D=\mathrm{char}^*\mathcal{M}_D. \] 
	\end{definition}

	\begin{example}
		Let $D$ be the diagram 
		\begin{center}
			\includegraphics[scale=1.25]{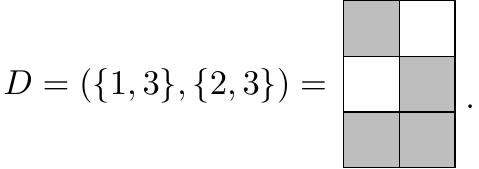}
		\end{center}
		Then the diagrams $C$ with $C\leq D$ are 
		\begin{center}
			\includegraphics[scale=1.25]{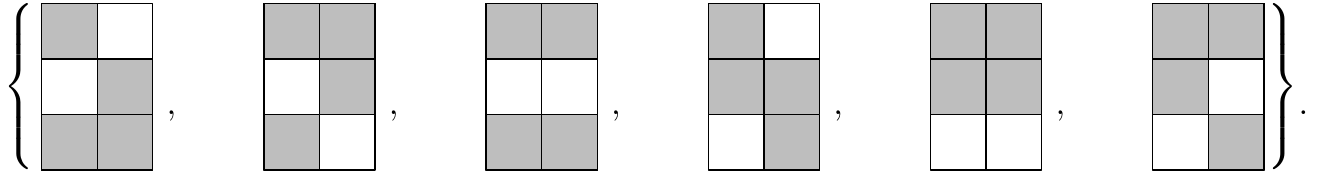}
		\end{center}
		The corresponding determinants are 	
		\[y_{11}y_{22}y_{33}^2, \quad 
		y_{11}y_{12}y_{23}y_{33} - y_{11}y_{13}y_{22}y_{33} \quad 
		y_{11}y_{12}y_{33}^2 \quad 
		y_{11}y_{22}y_{23}y_{33} \quad
		y_{11}y_{12}y_{23}^2 - y_{11}y_{13}y_{22}y_{23}\quad 
		y_{11}y_{12}y_{23}y_{33}\]
		These determinants are all linearly independent eigenvectors of $X$, so 
		\[\chi_D(x_1,x_2,x_3) = x_1x_2x_3^2+2x_1^2x_2x_3+x_1^2x_3^2+x_1x_2^2x_3+x_1^2x_2^2. \]
	\end{example}

	\begin{definition}
		For any diagram $D\subseteq [n] \times [n]$ with columns $D_1,\ldots,D_n$, we write $x^D$ for the monomial
		\[x^D = \prod_{j=1}^{n}\prod_{i\in D_j}x_i. \]
	\end{definition}

	The following two easy results describe the supports and coefficients of dual characters of diagrams.
	\begin{theorem}[cf. {\cite[Theorem 7]{FMS}}]
		\label{thm:monomials}
		For any diagram $D\subseteq [n]\times [n]$, the monomials appearing in $\chi_D$ are exactly 
		\[\left\{x^D\ \middle|\   C\leq D \right\}.\]
	\end{theorem}
	
	\begin{corollary}
		\label{cor:coefficients}
		Let $D\subseteq [n]\times [n]$ be a diagram. Fix any diagram $C^{(1)}\leq D$ and set $\bm{m}=x^{C^{(1)}}$.
%		\[\bm{m}=\prod_{j=1}^{n}\prod_{i\in C^{(1)}_j}x_i\] 
		Suppose $C^{(1)}, \ldots, C^{(r)}$ are all the diagrams $C\leq D$ such that $x^C=\bm{m}$. Then, the coefficient of $\bm{m}$ in $\chi_D$ is equal to the number of linearly independent polynomials over $\mathbb{C}$ among $\displaystyle\left\{\prod_{j=1}^{n}\det\left(Y_{D_j}^{C^{(i)}_j}\right) \ \middle|\  i\in [r] \right\}$. 
	\end{corollary}

%	The dual character of $\mathcal{M}_D$ has been shown in certain cases to be a Schubert polynomial \cite{KP} or a key polynomial \cite{flaggedLRrule}.

	\subsection{Schubert Polynomials}
	
	Recall the \emph{divided difference operators} $\partial_j$ for $j\in[n-1]$ are operators on the polynomial ring $\mathbb{C}[x_1,\ldots,x_n]$ defined by
	\[\partial_j(f)=\frac{f-(s_j\cdot f)}{x_j-x_{j+1}}
	=\frac{f(x_1,\ldots,x_n)-f(x_1,\ldots,x_{j+1},x_j,\ldots,x_n)}{x_j-x_{j+1}}.\]
	\begin{definition}
		The \emph{Schubert polynomial} $\mathfrak{S}_w$ of $w\in S_n$ is defined recursively on the weak Bruhat order. Let $w_0=n \hspace{.1cm} n\!-\!1 \hspace{.1cm} \cdots \hspace{.1cm} 2 \hspace{.1cm} 1 \in S_n$, the longest permutation in $S_n$. If $w\neq w_0$ then there is $j\in [n-1]$ with $w(j)<w(j+1)$ (called an \emph{ascent} of $w$). The polynomial $\mathfrak{S}_w$ is defined by
		\begin{align*}
		\mathfrak{S}_w=\begin{cases}
		x_1^{n-1}x_2^{n-2}\cdots x_{n-1}&\mbox{ if } w=w_0,\\
		\partial_j \mathfrak{S}_{ws_j} &\mbox{ if } w(j)<w(j+1).
		\end{cases}
		\end{align*}
	\end{definition}
	
	\begin{definition}
		The \emph{Rothe diagram} $D(w)$ of a permutation $w\in S_n$ is the diagram
		\[ D(w)=\{(i,j)\in [n]\times [n] \mid i<w^{-1}(j)\mbox{ and } j<w(i) \}. \]
	\end{definition}
	The diagram $D(w)$ can be visualized as the set of boxes left in the $n\times n$ grid 
	after you cross out all boxes weakly below $(i,w(i))$ in the same column, or weakly right of $(i,w(i))$ in the same row for each $i\in [n]$.
%	Note that Rothe diagrams have the \emph{northwest property}: If $(r,c'),(r',c)\in D(w)$ with $r<r'$ and $c<c'$, then $(r,c)\in D(w)$.
	
	\begin{example}
		If $w=31542$, then 
		\begin{center}
			\begin{tikzpicture}[scale=.55]
				\draw (0,0)--(5,0)--(5,5)--(0,5)--(0,0);
				\draw (1,0)--(1,5);
				\draw (2,0)--(2,5);
				\draw (3,0)--(3,5);
				\draw (4,0)--(4,5);
				\draw (0,1)--(5,1);
				\draw (0,2)--(5,2);
				\draw (0,3)--(5,3);
				\draw (0,4)--(5,4);
			%		\draw (5,4.5) -- (2.5,4.5) node {$\bullet$} -- (2.5,0);
			%		\draw (5,3.5) -- (0.5,3.5) node {$\bullet$} -- (0.5,0);
			%		\draw (5,2.5) -- (4.5,2.5) node {$\bullet$} -- (4.5,0);
			%		\draw (5,1.5) -- (3.5,1.5) node {$\bullet$} -- (3.5,0);
			%		\draw (5,0.5) -- (1.5,0.5) node {$\bullet$} -- (1.5,0);
				
				\fill[draw=black,fill=lightgray] (0,4)--(1,4)--(1,5)--(0,5)--(0,4);
				\fill[draw=black,fill=lightgray] (1,4)--(2,4)--(2,5)--(1,5)--(1,4);
				\fill[draw=black,fill=lightgray] (1,2)--(2,2)--(2,3)--(1,3)--(1,2);
				\fill[draw=black,fill=lightgray] (1,1)--(2,1)--(2,2)--(1,2)--(1,1);
				\fill[draw=black,fill=lightgray] (3,2)--(4,2)--(4,3)--(3,3)--(3,2);
				\node at (-1.5,2.5) {$D(w)=$};
				\node at (8.9,2.5) {$=(\{1\},\{1,3,4\},\emptyset,\{3\},\emptyset).$};
			\end{tikzpicture}
		\end{center}
	The Schubert polynomial of $w$ is computed by
	\[\mathfrak{S}_{w} = \partial_2\partial_1\partial_3\partial_2\partial_4 (x_1^4x_2^3x_3^2x_4). \]
	\end{example}

	Via Rothe diagrams, Schubert polynomials occur as special cases of dual characters of flagged Weyl modules:
	\begin{theorem}[\cite{KP}]
		\label{thm:kp}
		Let $w$ be a permutation with Rothe diagram $D(w)$. Then, 
		\[\mathfrak{S}_w = \chi_{D(w)}. \]
	\end{theorem}

	\subsection{Key Polynomials}
	
	Key polynomials were first introduced by Demazure for Weyl groups \cite{demazure}, and studied in the context of the symmetric group by Lascoux and Sch\"utzenberger in \cite{LS1,LS2}. Recall the \emph{key polynomial} $\kappa_{\alpha}$ of a composition $\alpha=(\alpha_1,\alpha_2,\ldots)$ is defined as follows.
	When $\alpha$ is a partition, $\kappa_\alpha=x^\alpha$.
	Otherwise, suppose  $\alpha_i<\alpha_{i+1}$ for some $i$. Then
	\[
	\kappa_\alpha= \partial_i (x_i \kappa_{\beta}), \ \ \text{where $\beta=(\alpha_1,\ldots,\alpha_{i+1},\alpha_{i},\ldots).$}
	\]
	 
	\begin{definition}[\cite{HHL, MTY}]
		Fix a composition $\alpha$, and set 
		\[l=\max\{i:\,\alpha_i\neq 0 \} \quad\mbox{and} \quad n=\max\{l,\alpha_1,\ldots, \alpha_l \}.\] 
		The \emph{skyline diagram} of $\alpha$ is the diagram $D(\alpha)\subseteq[n]\times[n]$ containing the leftmost $\alpha_i$ boxes in row $i$ for each $i\in[n]$.
	\end{definition}

	\begin{example}
		If $\alpha=(3,2,0,1,1)$, then
		\begin{center}
			\begin{tikzpicture}[scale=.55]
			\draw (0,0)--(5,0)--(5,5)--(0,5)--(0,0);
			\draw (1,0)--(1,5);
			\draw (2,0)--(2,5);
			\draw (3,0)--(3,5);
			\draw (4,0)--(4,5);
			\draw (0,1)--(5,1);
			\draw (0,2)--(5,2);
			\draw (0,3)--(5,3);
			\draw (0,4)--(5,4);
			%		\draw (5,4.5) -- (2.5,4.5) node {$\bullet$} -- (2.5,0);
			%		\draw (5,3.5) -- (0.5,3.5) node {$\bullet$} -- (0.5,0);
			%		\draw (5,2.5) -- (4.5,2.5) node {$\bullet$} -- (4.5,0);
			%		\draw (5,1.5) -- (3.5,1.5) node {$\bullet$} -- (3.5,0);
			%		\draw (5,0.5) -- (1.5,0.5) node {$\bullet$} -- (1.5,0);
			
			\fill[draw=black,fill=lightgray] (0,4)--(1,4)--(1,5)--(0,5)--(0,4);
			\fill[draw=black,fill=lightgray] (1,4)--(2,4)--(2,5)--(1,5)--(1,4);
			\fill[draw=black,fill=lightgray] (2,4)--(3,4)--(3,5)--(2,5)--(2,4);
			
			\fill[draw=black,fill=lightgray] (0,3)--(1,3)--(1,4)--(0,4)--(0,3);
			\fill[draw=black,fill=lightgray] (1,3)--(2,3)--(2,4)--(1,4)--(1,3);
			
			\fill[draw=black,fill=lightgray] (0,1)--(1,1)--(1,2)--(0,2)--(0,1);
			
			\fill[draw=black,fill=lightgray] (0,0)--(1,0)--(1,1)--(0,1)--(0,0);
			
			\node at (-1.5,2.5) {$D(\alpha)=$};
			\node at (9.6,2.5) {$=(\{1,2,4,5\}, \{1,2\}, \{1\}, \emptyset, \emptyset).$};
			\end{tikzpicture}
		\end{center}
	The key polynomial of $\alpha$ is computed by
	\[\kappa_{\alpha} = \partial_3(x_3\partial_4(x_4(x_1^3x_2^2x_3x_4))). \]
	\end{example}

	\begin{theorem}[\cite{keypolynomials}]
		\label{thm:key}
		Let $\alpha$ be a composition with skyline diagram $D(\alpha)$. Then		
		\[\kappa_\alpha=\chi_{D(\alpha)}. \]
	\end{theorem}

\section{A Lower Bound for $\chi_D(1,\ldots,1)$}
\label{sec:lower}
In this section, we prove a lower bound for the principal specialization of the dual character of any diagram. We then specialize this bound to Schubert and key polynomials.

\begin{definition}
Fix a diagram $D$. For each box $(i, j) \in D$, the \textit{rank} of that box is
\[
\textup{rank}_D(i, j) = \#\{ k \colon 1 \leq k \leq i \text{ and } (k, j) \notin D \}.
\]
The \textit{rank} of $D$ is
\[
\textup{rank}(D) = \sum_{(i, j) \in D} \textup{rank}_D(i, j).
\]
\end{definition}

%\begin{lemma} \acom{I don't think we even use this}
%If $C \leq D$ then $\textup{rank}(C) \leq \textup{rank}(D)$, and if $C < D$ then $\textup{rank}(C) < \textup{rank}(D)$.
%\end{lemma}
%\begin{proof}
%Let $C = (C_1, \dots, C_n)$ and $D = (D_1, \dots, D_n)$, so $C_j \leq D_j$ for all $j \in [n]$. Fix a column $j$. Then $C_j \leq D_j$ means we can write $C_j = \{ x_1, \dots, x_m \}$, $D_j = \{ y_1, \dots, y_m \}$, and $x_i \leq y_i$ for all $i \in [m]$. Note that
%\[
%\textup{rank}_C(x_i, j) = x_i - i \quad \text{and} \quad \textup{rank}_D(y_i, j) = y_i - i,
%\]
%so $\textup{rank}_C(x_i, j) \leq \textup{rank}_D(y_i, j)$ for all $i \in [m]$. Hence,
%\[
%\sum_{i \in C_j} \textup{rank}_C(i, j) \leq \sum_{i \in D_j} \textup{rank}_D(i, j).
%\]
%By summing across columns, we conclude that
%\[
%\sum_{(i, j) \in C} \textup{rank}_C(i, j) \leq \sum_{(i, j) \in D} \textup{rank}_D(i, j).
%\]
%
%Now, if $C < D$, then $C_j < D_j$ for at least one $j \in [n]$. For such a $j$, tracing the argument above tells us
%\[
%\sum_{i \in C_j} \textup{rank}_C(i, j) < \sum_{i \in D_j} \textup{rank}_D(i, j),
%\]
%and as a result
%\[
%\sum_{(i, j) \in C} \textup{rank}_C(i, j) < \sum_{(i, j) \in D} \textup{rank}_D(i, j).
%\]
%\end{proof}

\begin{lemma}
	\label{lem:chain}
	Let $D$ be a diagram and let $r = \textup{rank}(D)$. Then there are diagrams $C^0, C^1, \dots, C^{r-1}$ such that $C^0 < C^1 < \cdots < C^{r-1} < D$ and $\textup{rank}(C^k) = k$.
\end{lemma}
\begin{proof}
	If $r = 0$ then the chain consists of just $D$ and there is nothing to prove. Assume $r > 0$. We begin with the case that $D$ has a single nonempty column. Without loss of generality, we may write $D = (D_1) = (\{a_1, \dots, a_m\})$. Since $\textup{rank}(D) > 0$, $D_1 \neq \{1, \dots, m\}$. Let $k$ be the largest integer less than $a_m$ such that $k\notin D_1$. Choose $i$ so that $a_i = k + 1$ (which must exist by definition of $k$). Define
	\[
	C_1 = \left( D_1 \setminus \{a_i\} \right) \cup \{k\}.
	\]
	Then $C_1 < D_1$, and $\textup{rank}(C_1) = \textup{rank}(D_1) - 1$. By induction, the result follows whenever $D$ has a single nonempty column. Since
	\[
	\textup{rank}((D_1, \dots, D_n)) = \sum_{j \in [n]} \textup{rank}((D_j)),
	\]
	the general case follows from the single column case by performing the above construction to one column at a time.
\end{proof}

Recall the \emph{inverse lexicographic order} on monomials: $x^a<_{\textup{invlex}}x^b$ if there exists $1\leq i \leq n$ such that $a_j=b_j$ for $i+1\leq j\leq n$, and $a_i<b_i$.
\begin{lemma}
\label{lem:obvious}
If $C < D$, then $x^C \neq x^D$.
\end{lemma}
\begin{proof}
Let $C = (C_1, \dots, C_n)$ and $D = (D_1, \dots, D_n)$, so $C_j \leq D_j$ for all $j \in [n]$. Fix a column $j$. Then $C_j \leq D_j$ means we can write $C_j = \{ a_1, \dots, a_m \}$ and $D_j = \{ b_1, \dots, b_m \}$ with $a_i \leq b_i$ for all $i \in [m]$. Consequently,
\[
\prod_{i \in C_j} x_i \leq_{\textup{invlex}} \prod_{i \in D_j} x_i.
\]
Since $C < D$, we know $C_j < D_j$ for at least one $j \in [n]$. For any such $j$, we have
\[
\prod_{i \in C_j} x_i <_{\textup{invlex}} \prod_{i \in D_j} x_i,
\]
so
\[
x^C=\prod_{j=1}^n \prod_{i \in C_j} x_i <_{\textup{invlex}} \prod_{j=1}^n \prod_{i \in D_j} x_i=x^D.
\]
In particular $x^C\neq x^D$.
\end{proof}

\begin{namedtheorem}[\ref{thm:main}]
	For any diagram $D$, 
	\[\chi_D(1, \dots, 1) \geq \textup{rank}(D) + 1.\]
\end{namedtheorem}
\begin{proof}
By Theorem~\ref{thm:monomials},
\begin{align*} 
%\label{eq:_main_ineq_1}
\chi_D(1, \dots, 1) \geq \# \left\{ x^C \;\middle|\; C \leq D \right\}.
\end{align*}
By Lemma~\ref{lem:chain}, there exists a chain of $r=\textup{rank}(D) + 1$ diagrams $C^0 < C^1 < \cdots < C^{r - 1} < D$. Thus, by Lemma~\ref{lem:obvious}, 
\begin{align*} 
%\label{eq:_main_ineq_2}
\# \left\{ x^C \;\middle|\; C \leq D \right\} \geq \# \left\{ x^C \;\middle|\; C \in \{C^0, C^1, \dots, C^{r-1}, D\} \right\} = r+1=\textup{rank}(D) + 1.
\end{align*}
\end{proof}

By specializing Theorem \ref{thm:main} to Rothe diagrams, we obtain a new proof of Theorem \ref{thm:schubprinspec}:
\begin{corollary}[{\cite[Theorem 1.1]{anna132}}]
	\label{cor:annarelation}
	For any permutation $w\in S_n$, 
	\[\mathfrak{S}_w(1,\ldots,1)\geq 1+p_{132}(w), \]
	where $p_{132}(w)$ is the number of 132-patterns in $w$.
\end{corollary}
\begin{proof}
	It is enough to show that $p_{132}(w) = \textup{rank}(D(w))$. By viewing 132-patterns of $w$ graphically in $D(w)$, one easily observes that 132-patterns are in transparent bijection with tuples $(i,j,k)$ such that $(i,j)\in D(w)$, $1\leq k<i$, and $(k,j)\notin D(w)$. The quantity $\textup{rank}(D)$ exactly counts these tuples.
\end{proof}

By specializing Theorem \ref{thm:main} to skyline diagrams, we obtain an analogous result for key polynomials. For a composition $\alpha$, let $\mathrm{rinv}(\alpha)$ denote the set of \emph{right inversions} of $\alpha$, the pairs $i<j$ such that $\alpha_i<\alpha_j$.

\begin{corollary}
	\label{cor:keys}
	For any composition $\alpha$, 
	\[\kappa_\alpha(1,\ldots,1)\geq 1+\sum_{(i,j)\in \mathrm{rinv}(\alpha)} (\alpha_j-\alpha_i). \]
\end{corollary}

We now characterize the case of equality in Theorem \ref{thm:main}.

\begin{definition}
	Let $D$ be any diagram. A pair of boxes $(i,j),(i',j')\in D$ is called an \emph{unstable pair} if
	\begin{itemize}
		\item $\mathrm{rank}_D(i,j)\geq 1$;
		\item $\mathrm{rank}_D(i',j')\geq 1$;
		\item If $i=i'$ or $j=j'$, then $\mathrm{rank}_D(i,j)+\mathrm{rank}_D(i',j')\geq 3$.
	\end{itemize}
\end{definition}

\begin{proposition}
	A diagram $D$ satisfies $\chi_D(1, \dots, 1) = \textup{rank}(D) + 1$ if and only if $D$ does not contain an unstable pair.
\end{proposition}
\begin{proof}
	Suppose $D$ contains an unstable pair $\{(i,j),(i',j')\}$. A simple case analysis shows one can move boxes in $D$ upwards to create diagrams $C,C'\leq D$ of the same rank with $x^C\neq x^{C'}$. This implies $\chi_D(1,\ldots,1)\neq \mathrm{rank}(D)+1$.
	
	Assume $D$ contains no unstable pair. If $\mathrm{rank}(D)=0$, then the result follows. Pick $(i,j)\in D$ with $\mathrm{rank}_D(i,j)\geq 1$. If $\mathrm{rank}_D(i,j)>1$, then any other positive rank box would form an unstable pair with $(i,j)$. Hence $(i,j)$ is the only positive rank box of $D$, and the result follows easily.
	
	Suppose $\mathrm{rank}_{D}(i,j)=1$. To avoid unstable pairs, all other positive rank boxes of $D$ either lie in row $i$, or they all lie in column $j$. In either case, they must all have rank exactly 1. If all positive rank boxes of $D$ lie in column $i$, then one observes there is a unique diagram $C\leq D$ with rank $k$ for each $k=0,1,\ldots, \mathrm{rank}(D)$, implying the result.
	
	If all positive rank boxes of $D$ lie in row $j$, then one observes that all diagrams $C\leq D$ of a fixed rank have the same monomial $x^C$, and their determinants span an eigenspace of dimension one in the flagged Weyl module.
\end{proof}

We now relate equality in Theorem \ref{thm:main} with the question of when $\chi_D$ is zero-one. Recall a polynomial $f$ is called \emph{zero-one} if all nonzero coefficients in $f$ equal $1$.
\begin{proposition}
	If a diagram $D$ satisfies $\chi_D(1, \dots, 1) = \textup{rank}(D) + 1$, then $\chi_D$ is zero-one.
\end{proposition}
\begin{proof}
	In order for $\chi_D(1, \dots, 1) = \textup{rank}(D) + 1$, it must happen that all diagrams $C\leq D$ with a fixed rank induce the same monomial $x^C$ and have dependent determinants in the flagged Weyl module. Since all diagrams $C,C'\leq D$ with $x^C=x^{C'}$ must have the same rank, it follows that all eigenspaces in the flagged Weyl module of $D$ have dimension one.
\end{proof}

We now provide a conjectural characterization of diagrams $D$ such that $\chi_D$ is zero-one. Consider the six box configurations shown in Figure \ref{fig:01configs}. In each configuration, an $\times$ (red) indicates the absence of a box; a shaded square (gray) indicates the presence of a box; and an unshaded square (white) indicates no restriction on the presence or absence of a box.
\begin{definition}
	Let $D$ be any diagram. We say $D$ contains a \emph{multiplicitous configuration} if there are $r_1<r_2<r_3<r_4$ and $c_1<c_2$ so that $D$ restricted to rows $\{r_1,r_2,r_3,r_4\}$ and columns $\{c_1,c_2\}$ equals one of the configurations 
%	I, II, III, IV, V, VI 
	shown in Figure \ref{fig:01configs}, up to possibly swapping the order of the columns.
\end{definition}

\begin{figure}[ht]
	\begin{center}
		\includegraphics[scale=1.2]{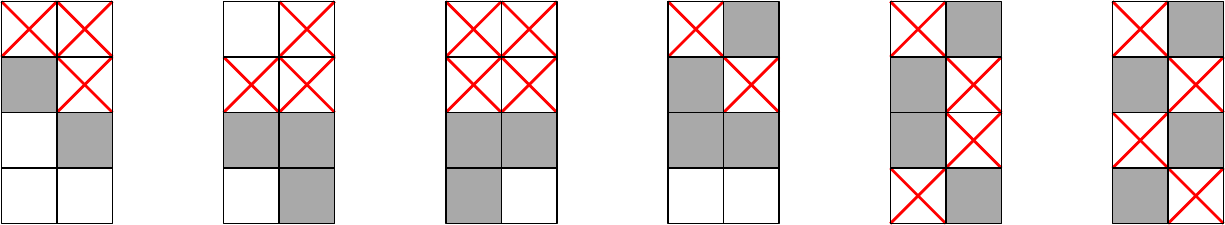}
	\end{center}
	\caption{The six multiplicitous configurations.}
	\label{fig:01configs}
\end{figure}

\begin{example}
	Consider the diagrams $D$ and $D'$ shown in Figure \ref{fig:expandnonexp}. The diagram $D$ does not contain instances of any multiplicitous configurations. The diagram $D'$ contains instances of each of the multiplicitous configurations.
\end{example}

\begin{figure}[ht]
	\begin{center}
		\includegraphics{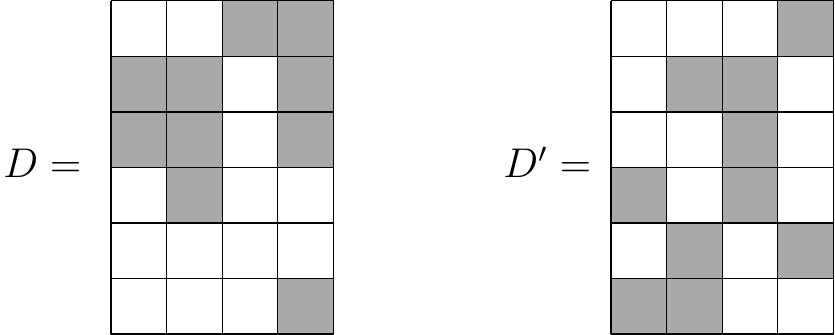}
	\end{center}
	\caption{}
	\label{fig:expandnonexp}
\end{figure}

\begin{proposition}
	\label{prop:01forward}
	If a diagram $D$ contains a multiplicitous configuration, then $\chi_D$ is not zero-one.
\end{proposition}
\begin{proof}
	It follows immediately from \cite[Theorem 5.8]{zeroone} that if the restriction of a diagram $D$ to rows $\{i_1,\ldots i_p\}$ and columns $\{j_1,\ldots,j_q\}$ equals a diagram $D'$, then the largest coefficient appearing in $\chi_D$ is bounded below by the largest coefficient appearing in $\chi_{D'}$. One easily checks that the dual characters of each of the multiplicitous configurations are not zero-one.
\end{proof}

\begin{conjecture}
	\label{conj:01backward}
	If $D$ is a diagram such that $\chi_D$ is not zero-one, then $D$ contains a multiplicitous configuration.
\end{conjecture}

Together, Proposition \ref{prop:01forward} and Conjecture \ref{conj:01backward} specialize to known results for Schubert and key polynomials: \cite[Theorem 1]{zeroone} when $D$ is the Rothe diagram of a permutation, and \cite[Theorem 1.1]{multfreekey} when $D$ is the skyline diagram of a composition.

\section{An Upper Bound for $\chi_D(1,\ldots,1)$}
\label{sec:upper}
In this final section, we recall a trivial upper bound for the principal specialization of the dual character of any diagram. We make a conjecture for the case of equality. From Corollary \ref{cor:coefficients}, it follows immediately that if $c_\alpha$ is the coefficient of $x^\alpha$ in $\chi_D$, then
\[c_\alpha \leq \# \{C\leq D \mid x^C=x^\alpha \}. \]
In particular, 
\[\chi_D(1,\ldots,1) \leq \#\{ C\mid C\leq D\}. \]

Fan and Guo gave the following characterization for equality when the diagram $D$  is northwest. Recall a diagram $D$ is \emph{northwest} if whenever $(i,j),(i',j')\in D$ with $i>i'$ and $j<j'$, one has $(i',j)\in D$.

\begin{theorem}[\cite{upperbound}]
	\label{thm:fanguoupperbound}
	For any northwest diagram $D$, 
	\[\chi_D(1,\ldots,1) = \#\{C\mid C\leq D \} \]
	if and only if $D$ contains no instance of the configuration shown in Figure \ref{fig:fanguoconfig}.
\end{theorem}

\begin{figure}[hb]
	\includegraphics{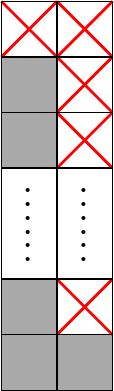}
	\caption{}
	\label{fig:fanguoconfig}
\end{figure}

We conjecturally extend Theorem \ref{thm:fanguoupperbound} to all diagrams. 
\begin{conjecture}
	Let $D$ be any diagram. Then $\chi_D(1,\ldots,1) = \# \{ C\mid C\leq D\}$ if and only if $D$ contains no instance of the configuration shown in Figure \ref{fig:maxconfig}.
\end{conjecture}

\begin{figure}[hb]
	\includegraphics{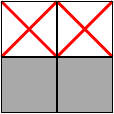}
	\caption{}
	\label{fig:maxconfig}
\end{figure}

%\section*{Acknowledgments} 
%We are grateful to Allen Knutson and Ricky Ini Liu for helpful discussions, and to Alex Fink for a careful reading.

\bibliographystyle{plain}
\bibliography{132-biblio}

\end{document}